\documentclass[12pt, twoside]{article}
\usepackage{amsmath,amsthm}
\usepackage{amssymb,latexsym}
\usepackage{enumerate}

\newtheorem{thm}{Theorem}[section]
\newtheorem{cor}[thm]{Corollary}
\newtheorem{lem}[thm]{Lemma}

\newtheorem{mainthm}[thm]{Main Theorem}

\theoremstyle{definition}

\numberwithin{equation}{section}

\begin{document}

\baselineskip=17pt


\title{Operators generated by Morse-Smale mappings}
\author{A. Antonevich\\
Institute of Mathematics, University of Bialystok\\
Akademicka 2, 15-267 Bialystok, Poland\\
E-mail: antonevich@bsu.by
\and J. Makowska\\
Institute of Mathematics, University of Bialystok\\
Akademicka 2, 15-267 Bialystok, Poland\\
E-mail: makowska@math.uwb.edu.pl
}

\date{}

\maketitle


\renewcommand{\thefootnote}{}

\footnote{2010 \emph{Mathematics Subject Classification}: Primary 

47B38; Secondary 47A56.}


\renewcommand{\thefootnote}{\arabic{footnote}}
\setcounter{footnote}{0}


\begin{abstract}
Weighted shift operators $B$ in space $L^2(X,\mu)$ that are induced by Morse-Smale type of mappings are considered. A description of the properties of $B-\lambda I$ for $\lambda$ belonging to spectrum $\Sigma(B)$ is given. In particular, there is the necessary and sufficient condition  that  $B-\lambda I$ be a one-sided invertible and the condition that  set $Im(B-\lambda I)$ be non-closed.  These conditions use a new notation: an oriented decomposition of  oriented graph $G(X,\alpha)$ generated by mapping $\alpha$. \end{abstract}

\section{Introduction}
A bounded linear  operator  $ B$ acting in  Banach space  $F(X)$ of
functions   (or vector-valued  functions) defined in a set  $X$ is called a
\emph{weighted shift operator } (WSO) if it can be represented as
 \begin{equation}
Bu(x)=a_0(x)u(\alpha(x)), \ \ x \in X, 
\label{w1}
\end{equation}
where  $\alpha: X \to X$  is a given mapping and 
 $a_0$ is a scalar  or matrix-valued  function on  $X$.

Such operators, the operator algebras generated by them and the functional equations related
to such operators have been studied by many authors in various functional spaces as
 independent objects and in relation to various applications. 
The properties of WSO depend on: the space of the functions under investigation,
the form of the coefficients and on map $\alpha$. The main problem, as pertains to this subject, is related to 
establishing the relations between the spectral properties of weighted shift operators and
the behavior of the trajectories, i.e. the dynamic properties of mapping $\alpha$ generating the operator.

In fact, a description  of spectra $\Sigma (B)$ in classical
spaces was obtained  in  a rather  general setting. These results
give us the  invertibility conditions for operator $ B
-\lambda I $.
Apart from  the invertibility conditions,  some subtle properties of
 operator $ B -\lambda I $  are of considerable interest, such
as closedness of the range and one-sided invertibility.

We restrict ourselves below to operators in space $L^2(X, \mu)$, so we formulate  results
known only for this case without dealing with their generalizations to other classes of spaces.
 
 The relationship  between the dynamics of
$\alpha $ and the subtle spectral properties of a WSO   has not been
investigated yet.  Earlier, such properties were studied only for
some special classes of maps $\alpha $.   One of these classes are the   diffeomorphisms of a closed interval where  only two fixed points exist: those attracting
and  repelling (\cite{karlowicz}, \cite{krawczenko}).  In a general situation the relationship between the dynamics of $\alpha$ and the spectral properties of $B-\lambda I$ has not been investigate yet. 
In \cite{ma_kwadrat}, \cite{ma_kostka}, as model examples, we shall consider mapping $\alpha$ on the $m-$dimensional simplex 
\begin{equation}X=  \{ x \in [0,1]^m:
0 \leqslant x_{1} \leqslant x_{ 2} \leqslant \ldots  \leqslant x_m
\leqslant  1\},\label{zbiorX}\end{equation}
where  $(m+1)$ fixed points exist:
 $$ F(k) = ( 0, 0,\ldots,0, \underbrace{1,1, \ldots,1}_k \ ), \ \ k
=0, 1,\ldots, m. $$  Point  $F(0) $ is repelling, point $F(m) $ is attracting and
all of the other points $ F(2), F(3) , \ldots, F(m-1)$ are saddle points. 
In the present paper we investigate WSO generated by  Morse-Smale type mappings $\alpha$. A mapping $\alpha$ will be called a map of Morse-Smale type if for $\alpha$ a set 
$Fix(\alpha)=\{F_i:\; \alpha(F_i)=F_i\}$ of periodic
points is finite and the trajectory of each point tends to a trajectory of one of the periodic
points when $j\to+\infty$, and tends to a trajectory of periodic points (possibly different ones)
when $j\to-\infty$. Mappings $\alpha$ from the model examples are particular cases of Morse-Smale type of mappings.

\section{The spectrum of weighted shift operator}  
For spaces $L^2(X,\mu)$ operator $B=a_0 T_\alpha$ can be conveniently written in the form of $B=a \widetilde{T}_\alpha$ where 
$$\widetilde{T}_{\alpha}u(x)=[\rho(x)]^{\frac{1}{2}} u(\alpha(x))$$
is a unitary operator in $L^2(X,\mu)$ and $\rho(x)=\frac{d\mu(\alpha^{-1}(x))}{d\mu(x)}$ is the Radon-Nikodym derivative. This function $\rho$ exists if and only if  map $\alpha$ preserves
the class of the measure $\mu$ (i.e. if for a measurable set $E$ equality $\mu((\alpha^{-1}(E))) = 0$ holds, if
and only if $\mu(E) = 0$).
When writing the operator in the form of $B = a\widetilde{T}_\alpha$, the properties of  operator $B$ are
more simply expressed using the reduced coefficient $a(x)=|a_0(x)| [\rho(x)]^{-\frac{1}{2}}$.
A description of the spectrum in  space $L^2(X,\mu)$ uses the following notions.
A measure $\mu$ on $X$ is called $\alpha$-invariant, if
$\mu(\alpha^{-1}(\omega))=\mu(\omega),$ for all measurable sets $\omega$.
A measure $\mu$ on $X$ is called ergodic with respect to $\alpha$, if from equality $\alpha^{-1}(\omega) = \omega$ follows
either $\mu(\omega)=0$ or $\mu(X\setminus\omega) = 0$.
For a given map $\alpha: X \to X$, a topological space $X$ is called $\alpha$-connected if cannot be
decomposed into two nonempty closed subsets invariant with respect to $\alpha$.

\begin{thm} [\cite{114},\cite{118},\cite{137}]
Let $X$ be a compact space, $\mu$ be a measure on $X$, such that its  support coincides with the whole space. Let $\alpha: X \to X $ be an invertible continuouse mapping, preserving the class of measure $\mu$. Let  $a\in C(X)$. Let $R(B)$ be the spectral radius of the operator $B=a\widetilde{T}_\alpha$. Then the following relation holds in $L^2(X,\mu)$:
$$R(B)=\max_{\nu \in M_{\alpha}(X)} \exp[ \int_X \ln
  |a(x)|d\nu],$$ 
where $a$ is the reduced coefficient and $M_\alpha(X)$ is the set of probability measures on $X$, invariant and ergodic with respect to the mapping $\alpha$.
If the set of nonperiodic point of the map $\alpha$ is everywhere dense in  $X$, and the space $X$ is $\alpha$-connected and $a(x)\neq 0$ for all $x$, then the spectrum $\Sigma(B)$ is a subset of the ring
     \begin{equation} 
     \Sigma(B) = \{ \lambda \in \mathbb{C}: {r}(B) \leq
|\lambda| \leq R(B)\}\label{ring} \end{equation}
where
 $$
   r(B) = \min_{\nu \in M_{\alpha}(X)} \exp[ \int_X \ln
  |a(x)|d\nu].
  $$
\label{postac_widma}
\end{thm}
For a  Morse-Smale type  mapping $\alpha$ a set $M_{\alpha}(X)$ consists of finitely many measures, each of which is concentrated
on the trajectory of one of the periodic points 
$$M_{\alpha}(X)=\{\delta_{F}, \; F\in Fix(\alpha)\}$$ where
$$\forall_{\omega\in \Phi}\quad \delta_{F}(\omega)=\left\{\begin{array}{ccc}
1,&\textrm{for}& F\in\omega;\\
0,&\textrm{for}& F\in X\setminus\omega.
\end{array}\right. $$ 
 In addition, for such a mapping the set
of nonperiodic points is dense in $X$ (with the exception of degenerate cases, when the fixed
points are isolated points of  space $X$). Therefore, for maps of the Morse-Smale type,
Theorem~\ref{postac_widma}  gives an explicit description of the spectrum.
\begin{cor}
 If $\alpha:X\to X$ is a Morse-Smale type of mapping,  $ a \in C(X)$  and $ a(x)
\ne 0 \; \forall_{x}$, then the spectrum $\Sigma(B)$ of  operator $B=a\widetilde{T}_\alpha$ in  space $ L^2(X,\mu)$ is a subset of the ring (\ref{ring})
  where $$
  R(B) = \max_{F \in Fix(\alpha)} |a(F)|, \  \  r(B) = \min_{F \in Fix(\alpha)} |a(F)|.
  $$
\label{wms}
\end{cor}

\section{Model examples. WSO on $m$-dimensional simplex $X\subset [0,1]^m$}
We shall   consider 
the  class of
mappings $ \alpha:X \to X $ of the form  
\begin{equation} 
\alpha(x_1,x_2,
 \ldots,x_n) = (\gamma(x_1),\gamma(x_2),  \ldots, \gamma(x_n)).\label{alfanaX}\end{equation}
   Where  $\gamma \,:[0,1]  \to [0,1]   $ be a diffeomorphism with only two fixed points: $0$ and $1$.
  To be exact, we  suppose that  $\gamma(x) > x$
   for   $0<x<1$.   
 
Our map  $ \alpha$ has in  $X$ only    $m+1$  fixed points $F(k),\;k\in\{0,1,\ldots, m\} $. 
 Let us arrange the numbers  $   |a(F(k))|$ in the increasing
   order and thus we obtain a transposition $ \sigma$ of the index-set  $\{ 0, 1, 2,
\ldots, m\} $ such that
   \begin{equation} |a(F(\sigma(0)))|< |a(F(\sigma(1)))| <
\ldots < |a(F(\sigma(m)))|.\label{trans}\end{equation}
The main results of this situation.

\begin{thm}[\cite{ma_kostka}, Theorem 4.2] Let   $X$ be simplex (\ref{zbiorX}). Let 
 $ \alpha:X \to X$ be a mapping given by  (\ref{alfanaX}) and  $B=a\widetilde{T}_\alpha $ is WSO
   in $L^2(X,\mu).$ Assume that
  $ a \in C(X),$ $ a(x) \ne 0 $ for all  $x \in X$, and   all the numbers  $   |a(F(k))|
 $ are different.  Let   $ \sigma$ be the transposition such that (\ref{trans}) holds.
 Assume also
that 
$$ |a(F(0))| < | a(F(m))| .$$
 The properties of the operator  $B- \lambda I  $
 depend on the disposition of  the numbers $
|a(F(k))|
 $ and  $ |\lambda|$  in the  following way.
\begin{enumerate}
\item[I.]  Let   
$$   |\lambda|=| a(F(m))| \ \ \text{ or} \ |\lambda|=|a(F(0))|.$$ 
Then the range  $Im(B - \lambda I )  $ is
nonclosed   subspace.

\item[II.] Let 
$$ |a(F(0))| < |\lambda|<| a(F(m))|  .$$
Then $\dim \ker(B - \lambda I)=\infty   $ and the range
$Im(B - \lambda I ) $ is an everywhere dense subspace in  $ L^2(X,\mu).$

The operator  $ B-\lambda I $ is right-sided invertible  if and only if  $$
 |\lambda|\ne | a(F(k))|  $$ and  the set  $\{ 0,1, 2, \ldots, k_0 \} $ is invariant under the transposition $ \sigma$,
where   $ k_0$ is such a number that
 $$ |a(F(\sigma(k_0)-1)| < |\lambda|<| a(F(\sigma(k_0))|.$$
 
 \item[III.] Let
$$ |a(F(\sigma(0)))|\leqslant  |\lambda|<|
a(F(0))|   \ \   \text{or } \ \    |a(F(m)|< |\lambda|\leqslant |
a(F(\sigma(m))|.$$
Then
   $Ker(B - \lambda I ) = 0  $ and   the range  $Im(B - \lambda I )  $ is a  nonclosed  everywhere dense subspace in $
  L^2(X,\mu).$

\end{enumerate}
\end{thm}
\section{Operators generated by Morse-Smale type mappings}

The main object of investigation in the present paper is a class of WSO $B=a \widetilde{T}_\alpha$ in space $L^2(X,\mu)$ where $X$ is $\alpha$-conected and a compact topological space, $\alpha:X\to X$ is an invertible, measurable Morse-Smale type mapping and a reduced coefficient $a\in C(X)$.

According to Corollary \ref{wms}, for $a(x)\neq 0$ spectrum  of the operator $B=a \widetilde{T}_\alpha$  is a ring
$$\Sigma(B)=\{\lambda\in\mathbb{C}:\;\min_{F\in Fix(\alpha)}|a(F)|\leq |\lambda|\leq \max_{F\in Fix(\alpha)}|a(F)|\}.$$ 
Circles 
$$S_k=\{\lambda:\; |\lambda|=|a(F)|,\; F\in Fix(\alpha)\}$$
belong to $\Sigma(B)$ and split the spectrum into smaller subrings. The spectral properties of $B-\lambda I$ are different for $\lambda$ belonging to every such   subring and every circle $S_k$. 
In paper \cite{ma_kwadrat} the model example of operators induced by a map of Morse-Smale type has been studied. Maps $\alpha$
have three fixed points, where the point $F_1$ is repelling, point $F_2$ is a saddle point and point $F_3$ is attracting. In this example the spectrum decomposes into two subrings and the properties of operator
$B-\lambda I$ depend on the domination/ minoration relationships between numbers $|a(F_k)|$ and $\lambda$. The
form of these inequalities is determined by the dynamics of map $\alpha$. Formulating the results even in that relatively simple case is very  cumbersome - it involves $18$ different subcases.  It turns out that the subtle properties of operator $B-\lambda I$ have a different dependence
on value $a(F_k)$ in fixed points of various types. Inorder  to describe the dynamics of $\alpha$ we use the structure of the oriented graph $G(X,\alpha)$. The vertices of $G(X,\alpha)$ are all fixed points $F\in Fix(\alpha)$ and the edge $F_j\to F_k$ belongs to $G(X,\alpha)$ if and only if there is a point $x\in X$ such
that its trajectory $\alpha^n(x)$ tends to $F_k$ when $n\to+\infty$ and tends to $F_j$ when $n\to-\infty$. In these terms the oriented edge $F_j\to F_k$ is included in the graph if and only if $\Omega_{k}^{+}\cap \Omega_{j}^{-}\neq\emptyset$, where
  $$\Omega_{k}^{+}=\{x: \; \lim_{n\to+\infty}\alpha^n(x)= F_k\}\cap \Omega_{j}^{-}=\{x:\; \lim_{n\to-\infty}\alpha^{n}(x)= F_j\}. $$
Useing a structured graph we can simply formulate the
important properties of the weighted shift operators.
By using the coefficient $a\in C(X)$ and the number  $\lambda\in\mathbb{C}$ we form two subsets of the set of vertices of
the graph
\begin{eqnarray}
G^{-}(\lambda,a)&=&\{F\in Fix(\alpha):\; |a(F)|<|\lambda|\},\nonumber\\
G^{+}(\lambda,a)&=&\{F\in Fix(\alpha): \;|\lambda|< |a(F)|\}.
\end{eqnarray}
We say that  $(G^{-}(\lambda,a), G^{+}(\lambda,a))$ give a decomposition of the graph, if the condition
 $$G(X,\alpha)=G^-(\lambda,a) \cup G^+(\lambda,a)$$
holds, i.e. if
\begin{equation}|\lambda|\neq |a(F)|,\; \forall_{F\in Fix(\alpha)}.\end{equation}
 
 The graph decomposition we call oriented to the right, if 
  any edge connecting the 
point $F_j\in G^-(a,\lambda)$ to the point $F_k\in G^+(a,\lambda)$ is oriented from $F_k$ to $F_j$.

The decomposition we call oriented to the left, if 
  any edge connecting the 
point $F_j\in G^-(a,\lambda)$ to the point $F_k\in G^+(a,\lambda)$ is oriented from $F_j$ to $F_k$.

The basic result in this direction is the following.

\begin{mainthm}
Let $\alpha$ be a Morse-Smale type of mapping. Let $X$ be a compact space $X$ and $a\in C(X)$. Let $B$ be a weighted shift operator. 
The operator $B-\lambda I$ is invertible from the right (left) if
and only if  the graph decomposition $(G^-(a, \lambda),G^+(a,\lambda))$ is  oriented to the right (to the left).
\label{glowne}
  \end{mainthm}
  \begin{cor}
  If there are $j$ and $k$, such that the set $\Omega_{k}^{+}\cap \Omega_{j}^{-}$
is dense in $X$,  then the image of the operator is closed if and only if the operator is 
one-sided invertible.
  \end{cor}

\section{Steps of proof of the Main Theorem.}
The Main Theorem may be proved in much the some way as Theorem 4.2 in \cite{ma_kostka}. Here are some important sketches of these concepts.  

\subsection{Reduction to operators $B_{kj}$.}
Consider a family of subsets 	$\Omega_{kj}$ in $X$, such that $\mu(\Omega_{kj})>0$. Then 
$$X=\bigcup_{k,j}\Omega_{kj}\;\; almost \;everywhere,$$
and $$ L^2(X)=\oplus_{kj}L^2(\Omega_{kj}).$$

Let  $B_{kj}$ denote the weighted shift operator in  space $L^2(\Omega_{kj})$. It follows that
 
$$B= \oplus_{kj} B_{kj}.$$
Therefor, an investigation of $B$ can reduced to an investigation of $B_{kj}$. The structures of operators $B_{kj}$ are similar, and it is sufficient to consider only one of them. We will consider operator $B_{kj}$ in $L^2(\overline{\Omega_{kj}},\mu_{kj})$ where $\overline{\Omega_{kj}}:=X_0$ is compact space and  measure $\mu_{kj}$ is such that
 $$\forall_{E\subset \overline{\Omega_{kj}}}\;\;\mu_{kj}(E)=\mu(E\cap\Omega_{kj}),$$
 then
  $$L^2(X,\mu)=\oplus_{kj}L^2(\overline{\Omega_{kj}},\mu_{kj}).$$
 This  reduction  is useful because the dynamics of $\alpha$ on $X_0$
is simpler  than the one
 on $X.$ Map $\alpha:X_0\to X_0$ has one repelling point, $F_j$, and one attracting point $F_k$, all other fixed points are saddle points.

\subsection{Isomorphism of the spaces $L^2(X_0,\mu)$ and $L^2(\Theta,l^2(\mathbb{Z}))$}
First of all we construct a representation of  $L_2(X_0,\mu)$ as a space
$ L^2( \Theta, l^2(\mathbb{Z}))$ of vector-valued functions
dependent on a parameter $ \tau  \in \Theta$, where
 $\Theta$ is a subset of $X_0$. In this representation    $B$ takes  the
form of  an operator-valued function $ B = B(\tau)$, where
\begin{equation}(B(\tau)u)(k)=a(\alpha^k(\tau)) u(k+1),
\label{dyskretny}
\end{equation} are
  weighted shift operators  in the space $ l^2(\mathbb{Z}).$ Thus
   our problem is reduced to the consideration of a
family of discrete weighted shift operators which can be studied
in detail.

Let us choose  a \emph{ fundamental set} for the map $\alpha,$
i.e. such
 a measurable set  $ \Theta$ that the sets $ \Theta_k
=\alpha_k(\Theta)$ are disjoint  and  $$\mu ( X_0 \setminus
\bigcup_{k \in \mathbb{Z}}\alpha_k(\Theta)) =0.$$

For an arbitrary map $\alpha$ it can happen   that the fundamental
set does not exist,  but for Morse-Smale type of mapping  such sets exists.
 \begin{lem}
Let $\alpha:X_0\to X_0$ be a map of the Morse-Smale type. Then  the fundamental set  $\Theta$ exists.  
\label{ist_fund}\end{lem}
\begin{proof}
 $V_j$, $V_k$ are open neighbourhoods of points $F_j$, $F_k$ such that $V_j\cap V_k=\emptyset$.  Let $\displaystyle{W=\bigcup_{j\geq 1}\alpha^j(V_k)}$ satisfy $\alpha(W)\subset W$ and $W\cap V_j=\emptyset$. 
If $x_0\in W$, then $\alpha(x_0)\in W$, $\alpha^2(x_0)\in W$ and so on. Let us denote $\Theta$ as a set of points  from trajectoriy, which is the first in  neighbourhood $W$, i.e.
  $$\Theta:=(W\setminus\alpha(W))\cap X_0.$$ 
 From this construcion set $\Theta$ is a measurable set, $\mu(\Theta)>0$ and 
 $$\alpha^j(\Theta)\cap\alpha^k(\Theta)=\emptyset,\; j\neq k.$$ 
\end{proof} 
The proofs of the assertions formulated below are in \cite{ma_kwadrat}, \cite{ma_kostka}.
\begin{lem}[Lemma 5.2 in \cite{ma_kostka}]
If the image   $ Im ( B-\lambda I) $
is closed  then  the image $ Im [ B(\tau) -\lambda I] $ is closed
 for almost all $ \tau \in \Theta$.
\end{lem}

\begin{lem}[Lemma 5.5 in \cite{ma_kostka}]
The operator  $ B-\lambda I $ is
right-sided invertible if and only if   the operators $ B(\tau)
-\lambda I$ are  right-sided invertible for almost all  $ \tau \in
\Theta$ and for every  such   $ \tau$ a right-sided inverse  $
R(\tau)$ to  $ B(\tau) -\lambda I$ can be chosen so  that the
family  $ R(\tau)$  will be bounded with respect to operator norm.

\end{lem}
From these lemmas we should find a family of the right-sided inverses operator to discrets operators $B(\tau)$ and estimates their norms.
\subsection{WSO  in $l^2(\mathbb{Z})$.}
\label{sd}
Let $Wu(k)=u(k+1)$ be the shift operator defined in the space $l^2
 (\mathbb{Z})$, $a = ( a(k))$ be a sequence and $ a W $ be the weighted shift operator in  $l^2
 (\mathbb{Z})$.
 We assume that   $ a(k)
\ne 0$  for all  $k$, and there exist the limits    
$$ \lim_{k \to
+ \infty} a(k) := a(+ \infty)\ne 0, \ \  \lim_{k \to - \infty}
a(k) := a(- \infty)\ne 0.$$

Then  spectrum  $ \Sigma (aW)$ is the ring $$\Sigma (aW)  = \{
\lambda: r \leqslant |\lambda| \leqslant R \},$$ where  $$ R =
\max \{ | a(+ \infty)|, |a(- \infty)|\}, \ \  r = \min \{ | a(+
\infty)|, |a(- \infty)|\}.$$

Assuming   that
$$  | a(- \infty)|<|\lambda| <|a(+ \infty)|$$
 the operator $ aW -\lambda I $ is right-sided invertible. 
 In \cite{ma_kwadrat} a   family of the right-sided inverses to $ aW -\lambda I $
was constructed and  estimates of their  norms were obtained in Lemma 7.1 and Lemma 7.2
\subsection{The dynamics of Morse-Smale type mappings.}
Let $\alpha:X_0\to X_0$ be a Morse-Smale type mapping and $V_i$ an open neighbourhood of the fixed point $F_i$. 

 \begin{lem}
Let 
 $$\widetilde{X_0}:=X_0\setminus\bigcup_{i}V_i.$$
There exists a number $N\in\mathbb{N}$ such that for every $\tau\in\Theta$ the number of the indeks where $\alpha^n(\tau)\in \widetilde{X_0}$ is no larger than $N$, i.e.
 $$card(\{\alpha^n(\tau)\}\cap \widetilde{X_0})\leq \textbf{N}\quad \forall_{\tau\in\Theta}.$$
 \label{skonczona}
 \end{lem}
 \begin{proof}
Let us denote  $V:=\bigcup_{i} V_i$. Set $V$ is open, then $\widetilde{X_0}=X_0\setminus V$ is closed. Set $X_0$ can be represented in the form
 $$X_0=\bigcup_{n=1}^\infty\bigcup_{i}\alpha^{-n}(V_i).$$
This is an open cover of $X_0$. While $\widetilde{X_0}\subset X_0$ is closed and compact, there  exists a finite subcover, i.e
 $$\widetilde{X_0}\subset \bigcup_{n=1}^\textbf{N}\bigcup_{i}\alpha^{-n}(V_i)\setminus \bigcup_i V_i.$$
We have $card(\{\alpha^n(\tau)\}\cap \widetilde{X_0})\leq \textbf{N}$.
 \end{proof}
  \begin{lem}
Let $S$ be an arbitrary natural number,   $F_a, F_b\in Fix(\alpha)$ and exist edge $F_a\to  F_b$. There exists an open set $W\subset X_0$, such that  for every $x\in W$
 \begin{equation}card(\{\alpha^j(x)\}\cap V_a)> S,\quad  card(\{\alpha^j(x)\}\cap V_b)> S,
 \label{wlS}
 \end{equation}
 and the points of trajectory $\alpha^j(x)$ after they leave neighbourhood $V_a$ but do not yet come to neighbourhood $V_b$ do not lay on the other neighbourhoods of  points $F_i\in Fix(\alpha)$. 
 \label{lematnaS}
 \end{lem}
 \begin{proof} 
Let  $Fix(\alpha)=\{F_a, F_b\}$ then $F_a$ is attracting point and $F_b$ is repelling point. There exists an open set $\displaystyle{W=\bigcup_{l\geq 1}\alpha^{-l}(V_b)}$  that condition (\ref{wlS}) holds.  For $x\in W$ we have: 
$$\lim_{n\to\infty}\alpha^n(x)=F_a\;\;\textrm{and}\;\; \alpha^{-n}(W)\subset W\;\; \textrm{and}\;\; \lim_{n\to\infty}\alpha^{-n}(x)=F_b.$$
Let $F_a, F_b\in Fix(\alpha)$ be saddle points. Choose open sets $W_a, W_b$ such that $W_a\subset V_a,\quad W_b\subset V_b$ and
$$ \forall_{x\in W_a}\;\;card\left(\{\alpha^j(x)\}\cap V_a\right)>S,\quad \forall_{x\in W_b}\;\;card\left(\{\alpha^j(x)\}\cap V_b\right)>S. $$
There exists an open set $W=W_a\cap \alpha^{-1}(W_b)$ that condition (\ref{wlS}) holds. 

\end{proof}

\subsection*{Acknowledgements}
This research was partly supported by NCN (Grant no. DEC-2011/01/B/ST1/03838) and suported by PFS (project UwB).


\begin{thebibliography}{HD}


\normalsize
\baselineskip=17pt

\bibitem{Antonevich}
A. Antonevich \emph{Linear Functional Equation. Operator Approach.} Birkhauser Verlag, Operator Theory Advances and Applications V. 83, 1996.
\bibitem{SF}
A. Antonevich, A. Buraczewski \emph{Dynamics of linear mapping and
invariant measure on sphere}, Demostratio Mathematica Vol. XXIX, No 4,
1996, 817-824.
\bibitem{7}  A. Antonevich, J. Jakubowska (Makowska)    \emph{ On the influence  of a saddle point on subtle spectral properties of weigthed shift
 operators.}  Vestn. Beloruss. Gos. Univ. Ser.1 Fiz. Mat. Inform. 3, 2006, 86 - 93.

\bibitem{ma_kwadrat} A.  Antonevich, J. Jakubowska (Makowska) \emph{
On model example of the weighted shift operator which generaited
by a mapping having one  saddle point},  Contemporary Mathematics.
Fundamental Directions 29, 2008, 29-48.

\bibitem{ma_kostka} A.  Antonevich, J. Makowska {\it On  spectral properties  of  weighted shift operators  generated by mappings with saddle points},
Complex Analysis and Operator Theory,   Volume 2, 2008.


\bibitem{karlowicz}   Karlovich A. Yu.,  Karlovich Yu.I.   \emph{ One
sided invertibility of binomial functional operators with a shift
in rearrangement-invariant spaces.} Integral Equations  Operator
Theory. 42, 2002, 201-228.
\bibitem{114} A. K. Kitover \emph{The spectrum of automorphisms with weight, and the Kamowitz-Scheinberg theorem.} Functional Analysis and Its
Appl. 13, 1979.
\bibitem{118} A. K. Kitover \emph{Spectral properties of homomorphisms with weight in algebras of continuous functions and their applications.} Zap. Nauchn. Sem. KOMI 107, 1982.

\bibitem{krawczenko}
   Kravchenko V. G., Litvinchuk G.S.
\emph{ Introduction to the Theory of Singular Integral Operators
with Shift.} Dordrecht.: Kluwer Acad. Publ.,  1994.
\bibitem{137} A. V. Lebedev \emph{Invertibility of elements in $C^\ast$- algebras generated by dynamical systems }, Uspekhi Mat. Nauk 34 (1979), no. 4, 199-200 (Russian). 
\bibitem{10}
J. Makowska  \emph{On  spectral properties of  weighted shift operators generated by mappings with saddle points in $L_p$},	Proceedings of the XXVII Workshop on Geometric Methods in Physics, American Institute of Physics, 2008, 96-101.
\bibitem{mostatni} 
J. Makowska \emph{On-side invertibility of the weighted shift operators},	Proceedings of the XXIX Workshop on Geometric Methods in Physics, American Institute of Physics, 2010, 101-105.

\bibitem{najmark}
M. A. Naimark \emph{Normed Rings}, $2$nd rev. ed. ,,Nauka'', Moscow, 1968.
\bibitem{252}W. C. Ridge \emph{Spectrum of a composition operator.} Proc. Amer. Math. Soc. 37, 1973.





\end{thebibliography}
\end{document}